\documentclass[11pt]{amsart}

\newif\ifdraft
\draftfalse

\usepackage[T1]{fontenc}
\usepackage[utf8]{inputenc}
\usepackage[margin=1.06in]{geometry}
\usepackage{amsmath,amsthm,amssymb,mathtools}
\usepackage[sc]{mathpazo}
\usepackage{microtype}
\usepackage{enumitem}
\usepackage{graphicx}
\usepackage{xcolor}
\usepackage[hidelinks]{hyperref}

\ifdraft
  \usepackage{showlabels}
  \usepackage{todonotes}
\else
  \newcommand{\todo}[2][]{}
\fi

\setlist[enumerate]{leftmargin=2.15em,itemsep=.25ex,topsep=.55ex}
\setlist[itemize]{leftmargin=1.85em,itemsep=.25ex,topsep=.55ex}
\allowdisplaybreaks

\theoremstyle{plain}
\newtheorem{theorem}{Theorem}[section]
\newtheorem{lemma}[theorem]{Lemma}
\newtheorem{proposition}[theorem]{Proposition}

\newtheorem{fact}[theorem]{Fact}
\newtheorem{claim}[theorem]{Claim}

\newtheorem*{theorem*}{Theorem}
\newtheorem*{lemma*}{Lemma}
\newtheorem*{proposition*}{Proposition}
\newtheorem*{corollary*}{Corollary}
\newtheorem*{claim*}{Claim}

\theoremstyle{definition}
\newtheorem{definition}[theorem]{Definition}

\newtheorem{remark}[theorem]{Remark}

\newenvironment{claimproof}[1][\proofname]
  {\proof[#1]}
  {\endproof}

\newcounter{step}

\newcommand{\FF}{\mathbb{F}}

\DeclareMathOperator{\aut}{Aut}

\newcommand{\forkindep}[1][]{%
  \mathrel{%
    \mathop{%
      \vcenter{%
        \hbox{\oalign{\noalign{\kern-.3ex}\hfil$\vert$\hfil\cr
              \noalign{\kern-.7ex}
              $\smile$\cr\noalign{\kern-.3ex}}}%
      }%
    }\displaylimits_{#1}%
  }%
}

\title{An automorphism tower of length $\kappa^+$}

\author{Yatir Halevi}
\address{Faculty of Mathematics, Technion - Israel Institute of Technology, Haifa, Israel}
\email{yatirh@math.technion.ac.il}

\author{Itay Kaplan}
\address{Einstein Institute of Mathematics, Hebrew University of Jerusalem, 91904, Jerusalem
Israel}
\email{kaplan@math.huji.ac.il}

\thanks{The first author was supported by The Israel Science Foundation (grant No. 295/26). The second author would like to thank the Israel Science Foundation for partial support of this
research (Grant no. 804/22).}

\date{}

\begin{document}

\begin{abstract}
For any infinite cardinal $\kappa$, we give an example of a centerless group $G$ of cardinality $\kappa$ whose automorphism tower terminates after $\kappa^+$ steps, answering a question of Simon Thomas. 
\end{abstract}

\maketitle

\section{Introduction and Preliminaries}

A centerless group $G$ embeds into its automorphism group $\aut(G)$ by mapping each element to its corresponding inner automorphism, thus identifying $G$ with $\mathrm{Inn}(G)\leq \aut(G)$. Since $\aut(G)$ is again a centerless group (\cite[Lemma 1.1.1]{ThomasBook}), we may construct the  automorphism tower: $G_0=G$, $G_{\alpha+1}=\aut(G_\alpha)$ and if $\alpha$ is a limit ordinal then $G_\alpha=\bigcup_{\beta<\alpha} G_\beta$. Thomas \cite{Thomas,thomasII} has shown that if $G$ is infinite then this procedure terminates after less than $(2^{|G|})^+$ steps, i.e. there is $\alpha< (2^{|G|})^+$ for which $G_\alpha=G_{\alpha+1}$.\footnote{In \cite{Thomas}, Thomas showed weak inequality and this was later improved independently by himself and Felgner, see \cite{thomasII} and \cite[Corollary 3.3.2]{ThomasBook}.} Later, \cite{KaplanShelah882} presented a simplified proof of Thomas' theorem (without using the axiom of Choice). To complete the historical picture, \cite{Wielandt} proved that if $G$ is finite, then the tower ends in finitely many steps (and his proof involves much deeper group theory). It is also worth mentioning that this problem can be formulated for groups with center where the maps are now homomorphisms, and indeed \cite{Hamkins} showed that even in this case the tower terminates, though there is no explicit bound. 
See \cite{ThomasBook} for more information on the history of the problem and earlier results.

For a centerless group $G$, let $\tau(G)$ be the minimal ordinal $\alpha$ such that $G_\alpha=G_{\alpha +1}$. In \cite{Thomas}, Thomas proved that for every cardinal $\kappa$ and ordinal $\alpha< \kappa^+$ there exists a centerless group $G$ with $|G|\leq \kappa$ and $\tau(G)=\alpha$.  In \cite[Question 1.9]{thomasII}, Thomas asked whether there exists an infinite centerless group $G$ with $\tau(G)\geq |G|^+$. This was answered positively (at least consistently) by Just, Thomas and Shelah in \cite{JuShTh} for uncountable cardinalities, but the question for countable groups remained open (see also \cite[Question 3.1.16]{ThomasBook}). We answer the original question in the affirmative (in ZFC). 

\begin{theorem*}[Theorem \ref{T:main}]
For any infinite cardinal $\kappa$, there exists a centerless group $G$ of cardinality $\kappa$ with $\tau(G)=\kappa^+$.
\end{theorem*}

In \cite{Sh:810}, Shelah proved (in ZFC) that if $\kappa$ is  strong limit singular of uncountable cofinality then there is a centerless group $G$ of cardinality $\kappa$ with $\tau(G)\geq 2^\kappa$. For general infinite cardinals $\kappa$, we do not have an example of a centerless group of cardinality $\kappa$ with $\tau(G)>\kappa^+$.

\subsection*{AI Statement.}  ChatGPT Sol 5.6 was asked whether it has an example of a countable centerless group $G$ with $\tau(G)=\omega_1$. This article grew out of its answer.

The original example that it gave was more complicated than the one we present here. 
The idea for considering the $S$-normalizer tower (see Definition \ref{D:tower}) was isolated from this original example. 
We dug further using ChatGPT and found that this technique was essentially already used to construct an example in a different field and under a different guise.

The paper itself was written by us.

\section{The Main Construction}
It seems that producing groups whose automorphism tower terminates after a prescribed number of steps is quite hard. In \cite{Thomas} (and \cite{JuShTh}), a procedure to create such a group was given by answering a different problem: Finding a group together with a subgroup whose normalizer tower terminates after a certain number of steps.

Let $G$ be any group and $H_0\leq G$ a subgroup. The normalizer tower of $H_0$ inside $G$ is defined in the following way: \[H_{\alpha+1}=N_G(H_\alpha)=\{a\in G: aH_\alpha a^{-1}=H_\alpha\}\] and if $\alpha$ is a limit ordinal then $H_\alpha=\bigcup_{\beta<\alpha} H_\beta$. 

We fix an infinite cardinal $\kappa$. Consider the symmetric group $\mathrm{Sym}(\kappa)$ endowed with the pointwise topology. 
\cite[Lemma 1.8]{JuShTh} gives:\footnote{\cite[Lemma 1.8]{JuShTh} is stated for automorphism groups of graphs of cardinality at most $\kappa$, but this can be translated to automorphism groups of first-order structures of cardinality $\kappa$ and whose language has cardinality at most $\kappa$, i.e., closed subgroups of $\mathrm{Sym}(\kappa)$. See, for example, \cite[Theorem 4.1.9]{ThomasBook}.}

\begin{fact}\label{F: normalizer to automorphism}
Let $G$ be a closed subgroup of $\mathrm{Sym}(\kappa)$ and $H_0\leq G$ a subgroup of cardinality $\leq \kappa$. If the normalizer tower of $H_0$ inside $G$ terminates after $\alpha$ steps then there exists a  centerless group $L$ of cardinality $\kappa$ with $\tau(L)=\alpha$. 
\end{fact}

The example we give is basically a normalizer tower inside a semidirect product. In order to ease the notation, we first translate this into the language of group actions. 

We start with the following
\begin{fact}\label{F:semidirect of nonarchPolish}
    Let $S$ be a discrete group of cardinality $\kappa$ and $N$ a closed subgroup of $\mathrm{Sym}(\kappa)$ with $S$ acting continuously by automorphisms on $N$. Then the topological group $N\rtimes S$ is isomorphic to a closed subgroup of $\mathrm{Sym}(\kappa)$. 
\end{fact}
\begin{proof}
    We use the following characterization of closed subgroups of $\mathrm{Sym}(\kappa)$, taken from \cite[Corollary 4.9]{ClosedofSym}: A topological group $G$ of weight $\lambda$ is isomorphic to a closed subgroup of $\mathrm{Sym}(\lambda)$ if and only if it is non-archimedean and Ra\u{i}kov complete. We give the definitions:
    \begin{list}{$\bullet$}{}
        \item The weight of a topological group is the least cardinality of a basis for its topology.
        \item A topological group is non-archimedean if it has a basis of the identity consisting of open subgroups.
        \item A topological group is Ra\u{i}kov complete if every net $\{x_\sigma\}_{\sigma\in \Sigma}$ satisfying that for every neighborhood of the identity $U\ni e$ there is $\sigma_0\in \Sigma$ such that for every $\sigma,\sigma'\geq \sigma_0$ both $x_\sigma x_{\sigma'}^{-1}$ and $x_{\sigma}^{-1}x_{\sigma'}$ are in $U$ is convergent.
    \end{list}

    Note that since we may embed $\mathrm{Sym}(\lambda)$ as a closed subgroup of $\mathrm{Sym}(\kappa)$ when $\lambda\leq \kappa$, \cite[Corollary 4.9]{ClosedofSym} can be relaxed to: A topological group $G$ of weight $\leq \kappa$ is isomorphic to a closed subgroup of $\mathrm{Sym}(\kappa)$ if and only if it is non-archimedean and Ra\u{i}kov complete.

    With the product topology, continuity of the action makes $N\rtimes S$ a topological group. 
    The weight of $S$ is $\kappa$ and the weight of $N$ is at most $\kappa$ since it is a closed subgroup of $\mathrm{Sym}(\kappa)$, so the weight of $N\rtimes S$ is at most $\kappa$.

    It is known that the semidirect product of two non-archimedean groups is again non-archimedean \cite[Corollary 3.4]{MeSh}. and that the semidirect product of two Ra\u{i}kov complete topological groups is again Ra\u{i}kov complete \cite[Theorems 12,17]{raikov}, nevertheless we add the proof for the sake of completeness.

    Let $\{V_t:t\in T\}$ be a basis of the identity of $e\in N$ consisting of open subgroups. Then $\{V_t\rtimes \{e\}:t\in T\}$ witnesses that $N\rtimes S$ is non-archimedean.
    
    We now show that $N\rtimes S$ is Ra\u{i}kov complete. Let $\{x_\sigma=(n_\sigma,s_\sigma)\}_{\sigma\in \Sigma}$ be a net in $N\rtimes S$ satisfying the assumption; we show it is convergent. Since $S$ is discrete, and hence Ra\u{i}kov complete, we have that eventually $s_\sigma=s$ for some $s\in S$, so we may assume $s_\sigma=s$ for all $\sigma$.  Thus we know that for every open neighborhood of the identity $U\times \{e\}$, both
 \[x_\sigma^{-1}x_{\sigma'}=(s^{-1}(n_\sigma^{-1}n_{\sigma'}),e) \text{ and}\]
    \[x_{\sigma}x_{\sigma'}^{-1}=(n_\sigma n_{\sigma'}^{-1},e)\] are eventually in $U\times\{e\}$. Consequently, as the action of $S$ is continuous, for any open neighborhood of the identity $e\in V\subseteq N$, eventually both  $n_\sigma^{-1}n_{\sigma'}$ and $n_\sigma n_{\sigma'}^{-1}$ are in $V$. As $N$ is  Ra\u{i}kov complete, $\{n_\sigma\}_{\sigma\in \Sigma}$ converges to some element $n\in N$, so $x_\sigma$ converges to $(n,s)$.
\end{proof}

\begin{lemma}\label{L: one step}
    Let $S$ and $N$ be groups with $S$ acting by automorphisms on $N$. Let $N_0\trianglelefteq N$ be a normal subgroup which is $S$-invariant.
    \begin{enumerate}
        \item $N_1=\{x\in N: \forall \sigma\in S,\, \sigma(x)x^{-1}\in N_0\}$ is an $S$-invariant subgroup of $N$ containing $N_0$.
        \item For $G=N\rtimes S$, $N_G(N_0\rtimes S)=N_1\rtimes S$.
    \end{enumerate}
\end{lemma}
\begin{proof}
    (1) Since $N_0$ is $S$-invariant, $N_1$ contains $N_0$. Since $N_0$ is a normal subgroup of $N$, $N_1$ is a subgroup. To show that it is $S$-invariant,  take $x\in N_1$ and $\sigma,\tau\in S$; we wish to show that $\sigma(\tau(x))\tau(x)^{-1}\in N_0$. As $N_0$ is invariant under $\tau$, this is equivalent to showing that $(\tau^{-1}\sigma\tau)(x)x^{-1}\in N_0$ but this is given by definition.

    (2) We use the following straightforward identity, whose proof we leave to the reader: If $A\rtimes B$ is a semidirect product then $A\cap N_{A\rtimes B}(B)=C_A(B)$, where we identify $A$ and $B$ with their corresponding subgroups of $A\rtimes B$.

    As $N_0$ is $S$-invariant, $N_0\trianglelefteq G$ and  $G/N_0=(N/N_0)\rtimes S$. Applying the above identity, $(N/N_0)\cap N_{G/N_0}(S)=C_{N/N_0}(S)$ and so $N_{G/N_0}(S)=C_{N/N_0}(S)\rtimes S$. By definition, $C_{N/N_0}(S)=N_1/N_0$ so after lifting $N_G(N_0\rtimes S)=N_1\rtimes S$.
\end{proof}

The previous lemma can be used to translate group actions into normalizer towers. 

\begin{definition}\label{D:tower}
    Let $N$ be an abelian group and $S$ a group acting on $N$ by automorphisms. For $N_0\leq N$ an $S$-invariant subgroup, we denote inductively on ordinals, $N_{\alpha+1}=\{x\in N:\forall \sigma\in S,\,\sigma(x)x^{-1}\in N_\alpha\}$ and if $\alpha$ is a limit ordinal set $N_\alpha=\bigcup_{\beta<\alpha} N_\beta$. 

    We call the sequence $(N_\alpha)_{\alpha\in \mathrm{Ord}}$ the \emph{$S$-normalizer tower of $N_0$ inside $N$}. We say that this tower terminates after $\alpha$ steps if $N_{\alpha}=N_{\alpha+1}$ but $N_\beta\subsetneq N_{\beta+1}$ for any $\beta<\alpha$.
\end{definition}

\begin{remark}\label{R: remarks on invariants tower}
    \begin{enumerate}
        \item By induction, $N_{\alpha}\subseteq N_{\alpha+1}$, for all $\alpha$.
        \item Also by induction, each $N_\alpha$ is $S$-invariant.

        \item If $S=\langle X\rangle$ is generated by some subset $X$ then for each $\alpha$, $N_{\alpha+1}=\{x\in N:\forall \sigma\in X,\,\sigma(x)x^{-1}\in N_\alpha\}.$ Indeed, we use the fact that $(\sigma\tau)(x)x^{-1}=\sigma(\tau(x)x^{-1})\sigma(x)x^{-1}$ (so it is a product of two elements from $N_\alpha$ by (2)) and that $\sigma^{-1}(x)x^{-1}=\sigma^{-1}((\sigma(x)x^{-1})^{-1})$.
        \item For any ordinal $\beta$, we have  $N_{\beta+1}/N_\beta=(N/N_\beta)^S$, where the latter is the subgroup of $S$-invariants. 
    \end{enumerate}
\end{remark}

The next proposition explains the choice of the name ``$S$-normalizer tower''.

\begin{proposition}\label{P:from action to automorphism}
Let $S$ be a discrete group of cardinality $\kappa$ and $N$ a closed abelian subgroup of $\mathrm{Sym}(\kappa)$ with $S$ acting continuously by automorphisms on $N$. Let $N_0\leq N$ be an  $S$-invariant subgroup of cardinality at most $\kappa$. 
    
    If the $S$-normalizer tower of $N_0$ inside $N$, $(N_\alpha)_{\alpha\in \mathrm{Ord}}$, terminates after $\alpha$ steps then there exists a centerless group $L$ of cardinality $\kappa$ with $\tau(L)=\alpha$. 
\end{proposition}
\begin{proof}
    Consider the group $G=N\rtimes S$, which is a closed subgroup of $\mathrm{Sym}(\kappa)$, and set $H_0=N_0\rtimes S$. By Remark \ref{R: remarks on invariants tower}(2), each $N_\alpha$ is $S$-invariant and is normal since $N$ is abelian. By Lemma \ref{L: one step}, $N_G(N_\alpha\rtimes S)=N_{\alpha+1}\rtimes S$. Setting $H_\alpha=N_\alpha\rtimes S$ for each $\alpha\in \mathrm{Ord}$, we obtain the normalizer tower of $H_0$ inside $G$. Finally, Fact \ref{F: normalizer to automorphism} yields the desired  centerless group.
\end{proof}

Before giving our example and thus proving the main theorem, we prove an upper bound for this procedure.

\begin{lemma}\label{L:discrete-constant above omega1}
    Let $S$ be a discrete group and $N$ a closed abelian subgroup of $\mathrm{Sym}(\kappa)$ with $S$ acting continuously by automorphisms on $N$. Let $N_0\leq N$ be an  $S$-invariant subgroup and $( N_\alpha)_{\alpha\in \mathrm{Ord}}$ its corresponding  $S$-normalizer tower.    If $|S|\leq \kappa$ then $N_{\kappa^+ +1}=N_{\kappa^+}$.
\end{lemma}
\begin{proof}
   Let $x\in N_{\kappa^+ +1}\setminus N_{\kappa^+}$. For any $\sigma\in S$, $\sigma(x)x^{-1}\in N_{\kappa^+}$, so there is some $\alpha_\sigma<\kappa^+$ such that $\sigma(x)x^{-1}\in N_{\alpha_\sigma}$. Let $\alpha=\sup_{\sigma\in S} \alpha_{\sigma}$; because $|S|\leq \kappa<\mathrm{cf}(\kappa^+)=\kappa^+$ one has $\alpha<\kappa^+$ and $\sigma(x)x^{-1}\in N_\alpha$ for all $\sigma \in S$. But then $x\in N_{\alpha+1}\subseteq N_{\kappa^+}$.
\end{proof}

The following is well known:
\begin{fact}\label{F:power as first order}
    Let $M$ be a first order structure and $S$ a set. Then $\aut(M)^S$, with the pointwise convergence topology,  is isomorphic as a topological group to the automorphism group of the structure $\coprod_{s\in S}M$ together with $S$-many unary predicates for the different copies of $M$.
\end{fact}

\begin{lemma}\label{L:enlarge to polish}
    Let $V$ and $P$ be groups, with $P$ a closed subgroup of $\mathrm{Sym}(\kappa)$, $\iota_0:V\hookrightarrow P$ a group embedding and $S$ a discrete group of cardinality at most $\kappa$ acting on $V$ by automorphisms. There exists a closed subgroup $G\leq\mathrm{Sym}(\kappa)$ together with a \emph{continuous} action of $S$ on $G$ by automorphisms, and an $S$-equivariant group embedding $\iota: V\hookrightarrow G$. Moreover, if $P$ is abelian then so is $G$.
\end{lemma}
\begin{proof}
     Consider $G=P^S$. Since closed subgroups of $\mathrm{Sym}(\kappa)$ are the same as automorphism groups of structures of cardinality $\kappa$ (in a language of cardinality at most $\kappa$), we may apply Fact \ref{F:power as first order} and so $G$ is isomorphic to a closed subgroup of $\mathrm{Sym}(\kappa)$.  If $P$ is abelian then so is $G$. 
     
     There is a natural action of $S$ on $G$ (``Bernoulli shift''): For any $f\in G$ and $t\in S$, $(t\cdot f)(s)=f(t^{-1}s)$. This gives a continuous action of $S$ on $G$ by automorphisms. 
    
    Define $\iota: V\to G$ by $\iota(v)(s)=\iota_0(s^{-1}v)$. It is a group embedding since $S$ acts on $V$ by automorphisms. We show that it is $S$-equivariant, i.e. for every $t\in S$ and $v\in V$, $t\cdot\iota(v)=\iota(t\cdot v)$. For any $s\in S$, $(t\cdot \iota(v))(s)=\iota(v)(t^{-1}s)=\iota_0((t^{-1}s)^{-1}\cdot v)=\iota_0((s^{-1}t)\cdot v)=\iota_0(s^{-1}\cdot(t\cdot v))=\iota(t\cdot v)(s)$.
\end{proof}

The example we give is a translation (and adaptation) of an example given by Hartley \cite[Proposition 3]{hartley}. His goal was to give an example of an uncountable artinian module over a countable ring. 
\begin{proposition}\label{P:Hartley-example}
    Let $F_\kappa$ be the free group on $\kappa$ generators as a discrete group. There exists a closed abelian subgroup $H\leq \mathrm{Sym}(\kappa)$, with a continuous action of $F_\kappa$ on $H$ by automorphisms,  such that for $H_0=\{e\}$ the $F_\kappa$-normalizer tower of $H_0$ inside $H$ terminates after $\kappa^+$ steps. 
\end{proposition}
\begin{proof}
    We follow the construction from \cite[Proposition 3]{hartley} (his construction was for $\aleph_1$ but it works for any cardinal). Let $V$ be a vector space over $\FF_2$ with basis $\{v_\alpha :\alpha<\kappa^+\}$. For any $0<\alpha<\kappa^+$ choose a surjection $f_\alpha:\kappa \to \alpha$. 

    Let $\{e_i: i<\kappa\}$ be the generators of $F_\kappa$. Define an action of $F_\kappa$ on the basis elements of $V$:
    \begin{flalign*}
       &  e_i\cdot v_\alpha=v_{f_\alpha(i)}+v_\alpha \text{, for $0<\alpha<\kappa^+$ and}\\
       & e_i\cdot v_0=v_0.
    \end{flalign*}
    Each such $e_i$ extends to a linear transformation of $V$. We show that they act as automorphisms. For each $\alpha$ and $i$, define $\alpha_0=\alpha$ and $\alpha_{n+1}=f_{\alpha_n}(i)$ (as long as $\alpha_n>0$); since $\alpha_{n+1}<\alpha_n$ unless $\alpha_n=0$, there is some integer $n_{i,\alpha}$ for which $\alpha_{n_{i,\alpha}}=0$. Let $W_{\alpha,i}=\mathrm{Span}\{v_{\alpha_0}, v_{\alpha_1},\dots,v_0\}$; it is a finite dimensional $e_i$-invariant subspace, and the matrix representation of $e_i$ with respect to this basis is triangular with $1$'s on the diagonal, so $e_i\restriction W_{\alpha,i}$ is an automorphism. Since $\alpha$ was arbitrary it follows that $e_i$ is both surjective and injective on $V$.
    
    For every $\alpha\leq \kappa^+$, set $V_\alpha=\mathrm{Sp}_{\FF_2}\{v_\beta:\beta<\alpha\}$, so $V_{\kappa^+}=V$.
   
    \begin{claim}
        Let $L_0=\{0\}$, $L_{\alpha+1}=\{x\in V:(e_i-1)x\in L_\alpha,\, \forall i\}$ and at limit stages we take a union. Then $L_\alpha=V_\alpha$ for every $\alpha<\kappa^+$; in particular $L_{\kappa^+}=V$.
    \end{claim}
    \begin{claimproof}
        We prove by induction, only the successor stage is required; assume that $L_\alpha=V_\alpha$.

        For the first direction, we show that $v_{\alpha}\in L_{\alpha+1}$. For that we need to show that for every $i$, $(e_i-1)v_{\alpha}\in L_\alpha$. For $\alpha=0$, $(e_i-1)v_0=0\in L_0$ for every $i$ and hence $v_0\in L_1$. For $\alpha>0$, $(e_i-1)v_\alpha=v_{f_{\alpha}(i)}\in V_\alpha=L_\alpha$ and therefore $v_\alpha\in L_{\alpha+1}$. Thus $L_{\alpha+1}$ contains $V_{\alpha+1}$.
        
        For the other direction, let $v=v_{\beta_1}+\dots+v_{\beta_k}$ be an element of $L_{\alpha+1}$ and assume that $\beta_1<\dots<\beta_k$. Assume towards contradiction that $\alpha <\beta_{k}$. Let $i<\kappa$ be such that $f_{\beta_k}(i) =\epsilon := \max\{\alpha,\beta_{k-1}\}$ (or just $\epsilon=\alpha$, if $k=1$). Then $(e_i-1) v = v_{\epsilon}+v_{f_{\beta_{k-1}}(i)} +\ldots +v_{f_{\beta_1}(i)}$, where terms with $\beta_j=0$ contribute $0$. Note that for any $j<k$ with $\beta_j>0$, $f_{\beta_j}(i)<\beta_j\leq \beta_{k-1}\leq \epsilon$ so the vectors after $v_{\epsilon}$ have smaller ordinal indices. But since $(e_i -1)v \in L_{\alpha} = V_{\alpha}$ this means that $\epsilon < \alpha$, contradiction.
      
    \end{claimproof}

    By Remark \ref{R: remarks on invariants tower}(3), $(L_\alpha)_{\alpha\in \mathrm{Ord}}$ is exactly the $F_\kappa$-normalizer tower of $L_0$ inside $V$. 
    
    Since $\dim_{\FF_2}V=\kappa^+$, we may embed the abelian group $V$ into the group $2^\kappa$: Indeed, the latter, seen as the $\FF_2$-vector space $(\FF_2)^\kappa$, has dimension $2^{\kappa}$ over $\FF_2$, so there is a linear embedding of $V$ into it. By Fact \ref{F:power as first order},  $2^\kappa$ is the automorphism group of a structure of cardinality $\kappa$ in a language of cardinality $\kappa$ (the group with two elements is the automorphism group of the structure with universe $2=\{0,1\}$ in the language of pure equality).
    
    By Fact \ref{L:enlarge to polish}, there exists an $F_\kappa$-equivariant embedding $\iota:V\hookrightarrow  H$, where $H$ is  an abelian closed subgroup of $\mathrm{Sym}(\kappa)$, and $F_\kappa$ acts on $H$ continuously by automorphisms.

    Let $\{0\}=H_0\leq H$ and let $(H_{\alpha})_{\alpha\in \mathrm{Ord}}$ be the  $F_\kappa$-normalizer tower  of $H_0$ inside $H$.
    
    \begin{claim}
        For any $\alpha\leq \kappa^+$, $\iota(V_\alpha)=H_\alpha\cap \iota(V)$.
    \end{claim}
    \begin{claimproof}
        We show by induction. Assume it is known for $\alpha$ and we show for $\alpha+1$ (the limit stage is obvious since $\iota$ is an embedding).

        For any $v\in V$, $\iota(v)\in H_{\alpha+1}$ if and only if for every $s\in F_\kappa$, $s\cdot \iota(v)-\iota(v)\in H_\alpha$. As $\iota$ is equivariant, it is equivalent to: For every $s\in F_\kappa$, $\iota(s\cdot v-v)\in H_\alpha$. By the induction hypothesis, this is equivalent to $s\cdot v-v\in V_\alpha$, i.e. $v\in V_{\alpha+1}$.
    \end{claimproof}

    We conclude that $H_{\alpha}\subsetneq H_{\alpha+1}$ for any $\alpha<\kappa^+$. By Lemma \ref{L:discrete-constant above omega1}, $H_{\kappa^++1}=H_{\kappa^+}$.
\end{proof}

\begin{theorem}\label{T:main}
    There exists a centerless group $G$ of cardinality $\kappa$ with $\tau(G)=\kappa^+$.
\end{theorem}
\begin{proof}
    Combining Proposition \ref{P:Hartley-example} and Proposition \ref{P:from action to automorphism}, there exists a centerless group $G$ of cardinality $\kappa$ with $\tau(G)=\kappa^+$.
\end{proof}
\bibliographystyle{alpha}
 \bibliography{tauG}

\end{document}